\documentclass[12pt,english]{article}
\usepackage[T1]{fontenc}
\usepackage[latin9]{inputenc}
\usepackage{amsthm}
\usepackage{amsmath}
\usepackage{amssymb}
\makeatletter
\theoremstyle{plain}
\newtheorem{thm}{Theorem}
  \theoremstyle{plain}
  \newtheorem{conjecture}[thm]{Conjecture}
  \theoremstyle{definition}
  \newtheorem{defn}[thm]{Definition}
  \theoremstyle{plain}
  \newtheorem{lem}[thm]{Lemma}
  \theoremstyle{plain}
  \newtheorem{prop}[thm]{Proposition}
    \theoremstyle{plain}
  \newtheorem{cor}[thm]{Corollary}
\newtheorem{claim}[thm]{Claim}

\usepackage{fullpage}
\usepackage{amsfonts}\usepackage{amsthm}\@ifundefined{definecolor}
 {\usepackage{color}}{}
\usepackage{epsf}\usepackage[labelformat=empty]{subfig}

\newcommand{\ep}{\epsilon}
\newcommand{\g}{\gamma}
\newcommand{\floor}[1]{\left\lfloor#1\right\rfloor}
\newcommand{\ceiling}[1]{\left\lceil#1\right\rceil}

\newcommand{\n}{2\times 10^8}

\makeatother

\usepackage{babel}

\begin{document}

\author{Phong Ch\^au %
\thanks{Department of Mathematics and Computer Science, Glendale Community College, Glendale AZ 85302, USA. e-mail: phong.chau@gcmail.maricopa.edu%
} \and Louis DeBiasio %
\thanks{School of Mathematical and Statistical Sciences, Arizona State University,
Tempe, AZ 85287, USA. e-mail address: louis@mathpost.asu.edu %
}\and H.A. Kierstead %
\thanks{School of Mathematical and Statistical Sciences, Arizona State University,
Tempe, AZ 85287, USA. e-mail address: kierstead@asu.edu;
research supported in part by NSF
grant DMS-0901520. %
}}

\title{P\'osa's Conjecture for graphs of order at least $\n$}
\maketitle
\begin{abstract}
In 1962 P\'osa conjectured that every graph $G$ on $n$ vertices with
minimum degree $\delta(G)\ge\frac{2}{3}n$ contains the square of
a hamiltonian cycle. In 1996 Fan and Kierstead proved the path version
of P\'osa's Conjecture. They also proved that it would suffice to show
that $G$ contains the square of a cycle of length greater than $\frac{2}{3}n$.
Still in 1996, Koml\'os, S\'ark\"ozy, and Szemer\'edi proved P\'osa's Conjecture,
using the Regularity and Blow-up Lemmas, for graphs of order $n\ge n_{0}$, where $n_{0}$ 
is a very large constant. Here we show without using these lemmas
that $n_{0}:=\n$ is sufficient.  We are motivated by the 
recent work of Levitt, Szemer\'edi and S\'ark\"ozy, but our methods are
based on techniques that were available in the 90's. 
\end{abstract}

\section{Introduction}

The \emph{square} $H^{2}$ of a graph $H$ is obtained by joining
all pairs $\{x,y\}\subset V(H)$ with distance $dist(x,y)=2$ in $H$.
If $H$ is a path (cycle) then $H^{2}$ is called a \emph{square}
path (cycle). Now fix a graph $G=(V,E)$ on $n$ vertices. We say
that $v_{1}\dots v_{t}$ is a \emph{square} path (cycle) in $G$ if
$v_{1}\dots v_{t}$ is a path (cycle) in $G$ and its square is contained
in $G$. In 1962 P\'osa \cite{E} conjectured: 
\begin{conjecture}
Every graph $G$ with $\delta(G)\ge\frac{2}{3}|G|$ contains a hamiltonian
square cycle. 
\end{conjecture}
During the 90's there were numerous partial results on P\'osa's conjecture.
Here we review a number that have a direct impact on this paper. Fan
and Kierstead \cite{FK1,FK2,FK3} proved the following three theorems. 
The first is a connecting lemma that
immediately yields an approximate version of P\'osa's conjecture. 
\begin{thm}[Fan
and Kierstead \cite{FK1}]
\label{thm:FK1}For every $\ep>0$ there exists a constant $m$ such that for
every graph $G$ with $\delta(G)\ge(\frac{2}{3}+\ep)|G|+m$
and every pair $e_{1},e_{2}$ of disjoint ordered edges, $G$ has a hamiltonian 
square path starting with $e_{1}$ and ending with $e_{2}$. In particular,
$G$ has a hamiltonian square cycle.
\end{thm}
We shall need two ideas from this paper---\emph{weak reservoirs}
\footnote{The term reservoir is not mentioned in \cite{FK1}, and the modifiers
\emph{weak}, \emph{strong }and \emph{special} are our own invention.
However, in light of more recent papers this terminology provides
a consistent transition (see Definition \ref{def:Reservoir}).
}, and \emph{optimal} square paths and cycles---which will be presented
in the next section. Roughly, given a graph $G$ on $n$ vertices, a weak reservoir is a small fraction $R$ of the vertex set $V(G)$  such that  $|N\cap R|\approx |N||R|/n$ for any neighborhood $N:=N(v)$.  Weak reservoirs were used to connect long square 
paths contained in  $V(G)\setminus R$. The second theorem is a
path version of P\'osa's Conjecture. 
\begin{thm}[Fan and Kierstead \cite{FK2}]
\label{thm:FK2}Every graph $G$ with $\delta(G)\ge\frac{2|G|-1}{3}$
contains a hamiltonian square path. 
\end{thm}
The third theorem shows that $V(G)$ can be partitioned into at most two square cycles.
\begin{thm}[Fan and Kierstead \cite{FK3}]
\label{thm:FK3}Suppose $G$ is a graph with $\delta(G)\ge\frac{2}{3}|G|$.
If $G$ has a square cycle of length greater than $\frac{2}{3}|G|$
then $G$ has a hamiltonian square cycle. Moreover, $V(G)$ can be
partitioned into at most two square cycles, each of length at least $\frac{1}{3}|G|$. 
\end{thm}
The proofs of Theorems \ref{thm:FK2} and \ref{thm:FK3} are based
on optimal paths and cycles, but do not use weak reservoirs. Theorem
\ref{thm:FK3} is essential to this paper, because it allows our constructions 
to terminate as soon as we get a square cycle of length greater than
$\frac{2}{3}|G|$.

Next came a major breakthrough. Koml\'os, S\'ark\"ozy and Szemer\'edi proved
their famous Blow-up Lemma \cite{KSSbu}, and used it and the Regularity
Lemma \cite{Sz} to prove: 
\begin{thm}[Koml\'os, S\'ark\"ozy and Szemer\'edi \cite{KSSp}]
\label{thm:KSSp}There exists a constant $n_0$ such that every graph
$G$ with $|G|\ge n_0$ and $\delta(G)\ge\frac{2}{3}|G|$ has a hamiltonian
square cycle.
\end{thm}
Their proof has the following structure. First they determine extremal
configurations that are very close to being counterexamples, but because
of the tightness of the degree condition, cannot achieve this status.
(For example, if the independence number $\alpha(G)>\frac{1}{3}|G|$
then $G$ does not have a hamiltonian square cycle, but then also does 
not satisfy $\delta(G)\ge\frac{2}{3}|G|.$ Moreover if $G$ has an
almost independent set of size almost $\frac{1}{3}|G|$ and $\delta(G)\geq \frac{2}{3}|G|$, then we will 
see that $G$ does have a hamiltonian square cycle.) Next they proved
that if $|G|$ is sufficiently large, $\delta(G)\ge\frac{2}{3}|G|$,
and $G$ has an extremal configuration, then $G$ has a hamiltonian
square cycle. When there are no extremal configurations, the Regularity
Lemma imposes a pseudo random structure on the graph that can be exploited,
using this lack of extremal configurations and the Blow-up Lemma,
to construct a hamiltonian square cycle. The use of the Regularity
Lemma causes the constant $n_0$ to be extremely large.

Very recently R\"odl, Ruci\'{n}ski and Szemer\'edi have made another 
important advance \cite{RRS1,RRS2}. They proved
the following version of Dirac's Theorem for 3-uniform hypergraphs
(3-graphs). An \emph{open chain} $P:=v_{1}v_{2}v_{3}\dots v_{s-2}v_{s-1}v_{s}$
in a 3-graph $H$ is a sequence of distinct vertices such that 
$v_{i}v_{i+1}v_{i+2}\in E(H)$ for all $i\in[s-2]$; $P$ is a \emph{closed
chain} if in addition $v_{s-1}v_{s}v_{1},v_{s}v_{1}v_{2}\in E(H)$.
\begin{thm}[R\"odl, Ruci\'{n}ski and Szemer\'edi \cite{RRS2}]
There exists an integer $n_{0}$ such that for every $3$-graph $H$
on at least $n_{0}$ vertices, if every pair of vertices of $H$ is
contained in at least $\lfloor\frac{1}{2}|H|\rfloor$ edges of $H$
then $H$contains a hamiltonian closed chain. 
\end{thm}
The remarkable proof is very long, but has a similar structure to
the proof of Theorem~\ref{thm:KSSp}. However, a major difference 
is that the non-extremal case does not use any version of the Blow-up Lemma, and regularity
(weak hypergraph regularity) is only used in a quite generic way to
construct various \emph{strong reservoirs}---weak reservoirs with no extreme sets. The Blow-up Lemma is replaced
by a construction based on an ingenious \emph{absorbing path }lemma\emph{,}
and a \emph{connecting} lemma, that uses the strong reservoir.

Levitt, S\'ark\"ozy and Szemer\'edi \cite{LSS} applied similar techniques 
to the non-extremal case of P\'osa's Conjecture without using the Regularity Lemma, and thus proved the result for much smaller graphs than those considered in Theorem \ref{thm:KSSp}. 

Here we show that P\'osa's Conjecture holds for graphs of order at least
$\n$ without using the Regularity-Blow-up
method. In addition, our proof of the extremal case holds for all $n$.
We were influenced by the ideas of \cite{LSS}, but only rely on results
from \cite{FK1,FK2,FK3}, and the idea from \cite{KSSp} of dividing
the problem into an extremal case and a non-extremal case.  We avoid the Blow-up Lemma and absorbing paths by using Theorem~\ref{thm:FK3}. 
 Our approach is explained fully in the
next section.

\subsubsection*{Notation}

Most of our notation is consistent with Diestel's graph theory text
\cite{Di}. In particular note that $P^{n}$ is a path on $n$ edges,
$|G|=|V(G)|$, $\|G\|=|E(G)|$, and $d(v)$ is the degree of the vertex
$v$. For $A,B\subseteq V(G)$, let $\|A,B\|=|E(A,B)|$, where $E(A,B)$ is the set of edges with one end
in $A$ and the other in $B$, in particular we shall write $\|a,B\|$ if $A=\{a\}$. We also use $\overline{\|A,B\|}$ to denote the number of edges  in the complement of $G$ that have one end in  $A$ and the other in  $B$. 
For $a_{1},a_{2},\dots,a_{k}\in V(G)$, let $N(a_{1},a_{2},\dots,a_{k})=N(a_{1})\cap N(a_{2})\dots\cap N(a_{k})$.

\section{Main theorem and proof strategy}

Here is our main result:
\begin{thm}
\label{main} Let $G$ be a graph on $n$ vertices with $n\geq n_{0}:=\n$.
If $\delta(G)\geq\frac{2}{3}n$, then $G$ has a hamiltonian square cycle.
\end{thm}
In this section we organize the structure of the proof. The first
step is to define a usable extremal configuration. Our choice is simpler
than the choice in \cite{LSS}, which was much simpler than the several
extremal configurations used in \cite{KSSp}. A priori, this makes the extremal case easier and the non-extremal case harder.
\begin{defn}
Let $G$ be a graph on $n$ vertices.  A set $S\subseteq V(G)$ is $\alpha$-\emph{extreme} if $|S|\geq (1-\alpha)\frac{n}{3}$
and $\|v,S\|<\alpha \frac{n}{3}$ for all $v\in S$. 
\end{defn}
The proof divides into two parts, depending on whether $G$ is $\frac{1}{36}$-extreme,
i.e., contains an $\alpha$-extreme set with $\alpha:=\frac{1}{36}$.
The extreme case is handled in Section~\ref{sec:Ex}, where we prove
the following theorem without assuming anything about the order of
$G$. Its proof only requires elementary graph theory. Notice that
$K_{3t+2}-E(K_{t+1})$ demonstrates that the degree condition is tight.
\begin{thm}[Extremal Case]
\label{thm:good}Let $G$ be a graph on $n$ vertices with $\delta(G)\geq \frac{2}{3}n$. If $G$ has a $\frac{1}{36}$-extreme set, then $G$ has a hamiltonian square cycle. 
\end{thm}
The non-extremal case is more complicated. In Section~\ref{sec:nex} we will prove: 
\begin{thm}[Non-extremal Case]
\label{non-extremal} Let $G$ be a graph on $n$ vertices with $\delta(G)\geq\frac{2}{3}n$
and $n\geq n_{0}:=\n$. If $G$ does not contain a $\frac{1}{36}$-extreme
set, then $G$ has a hamiltonian square cycle. 
\end{thm}

Note that if $G$ has an $\alpha$-extreme set $S\subseteq V(G)$ for some $\alpha<\frac{1}{36}$, then $S$ is a $\frac{1}{36}$-extreme set.  This explains why we only consider $\frac{1}{36}$-extreme sets in Theorems \ref{thm:good} and \ref{non-extremal}.

The proof of Theorem~\ref{non-extremal} has three parts. First we
use the Reservoir Lemma (Lemma~\ref{reservoir}) to construct a special reservoir
$R$ with $|R|<\frac{1}{3}n$. Then we use the Path Cover Lemma (Lemma~\ref{pathcover}) to construct two disjoint square
paths $P_{1},P_{2}$ in $G-R$ such that $|P_{1}|+|P_{2}|>\frac{2}{3}n$
using techniques and results from \cite{FK1,FK2}. Finally, we use
the properties of the special reservoir $R$, together with our version
of the Connecting Lemma (Lemma~\ref{connecting}), to connect the ends of
$P_{1}$ to the ends of $P_{2}$ by disjoint square paths in $R$
so as to form a square cycle of length greater than $\frac{2}{3}n$.
Thus by Theorem~\ref{thm:FK3} we obtain a hamiltonian square cycle.

\subsection{Reservoirs and the Connecting Lemma}

The bottleneck in this line of attack is in determining properties for
special reservoirs that are strong enough to prove the Connecting
Lemma, yet weak enough to ensure the existence of special reservoirs
in moderately sized graphs. 
In the process of constructing a connecting square path we need to know that 
certain subsets of the reservoir are nonextreme. Since it is too expensive to ensure that all subsets are nonextreme,
we anticipate a limited collection of  \emph{special} subsets that might appear in this construction, and construct
 a reservoir with no extreme special sets.
\begin{defn}
\label{def:special}A set $S\subseteq V(G)$ is \emph{special} if there
exist (not necessarily distinct) vertices $u,v,w,x,y\in V(G)$ such that $S=(N(u,v,w)\cup N(u,v,x))\cap N(y)$.
\end{defn}

A set $S$ of size at least $(1-\alpha)\frac{n}{3}$ that is not $\alpha$-extreme has at least one vertex with ``large'' degree to $S$, but we will need more than one vertex of ``large'' degree, so we define a more general notion of extremity.

\begin{defn}\label{def:alpha-beta}
Let $G$ be a graph with $n$ vertices.  A set $S\subseteq V(G)$ is $(\alpha,\beta)$-\emph{extreme} if $|S|\geq (1-\alpha+\beta)\frac{n}{3}$ and there are fewer than $\floor{\beta\frac{n}{3}}$ vertices $v\in S$ such that $\|v,S\|\geq \alpha \frac{n}{3}$. 
\end{defn}

So a set $S$ of size at least $(1-\alpha+\beta)\frac{n}{3}$ that is not $(\alpha,\beta)$-extreme has at least $\floor{\beta\frac{n}{3}}$ vertices with ``large'' degree to $S$.  In the non-extremal case we know that $G$ contains no $\alpha$-extreme sets, but we must ensure for the Connecting Lemma that the reservoir has no $(\alpha', \beta')$-extreme special sets.  So we use the following simple observation when constructing the reservoir.

\begin{lem}
\label{nicevertices} Let $G$ be a graph on $n$ vertices and let $\alpha,\beta>0$. If $G$ has no $\alpha$-extreme sets and $S\subseteq V(G)$ with $|S|\geq (1-\alpha+\beta)\frac{n}{3}$, then $S$ is not $(\alpha, \beta)$-extreme. 
\end{lem}
\begin{proof}
Suppose $S$ is $(\alpha, \beta)$-extreme and let $S'=\{v\in S: \|v, S\|\geq \alpha\frac{n}{3}\}$. Since $S$ is $(\alpha, \beta)$-extreme, we have $|S'|<\floor{\beta\frac{n}{3}}$.  Thus $|S\setminus S'|\geq (1-\alpha)\frac{n}{3}$ and $\|v, S\setminus S'\|<\alpha\frac{n}{3}$ for all $v\in S\setminus S'$, contradicting the fact that $G$ has no $\alpha$-extreme sets.
\end{proof}

Here are the technical definitions of $(\ep, \varrho)$-weak,
$(\alpha, \ep, \varrho)$-strong and $(\alpha, \beta, \ep, \varrho)$-special
reservoir. 
\begin{defn}
[Reservoir] \label{def:Reservoir}Let $G$ be a graph on $n$ vertices. Let $1\geq \varrho\geq 0$ and
$\ep>0$. An \emph{$(\ep, \varrho)$-weak reservoir} is a set $R\subseteq V(G)$
such that $|R|=\ceiling{\varrho n}$ and for all $u\in V(G)$, \[
\left(\frac{d(u)}{n}-\ep\right)|R|\le\|u,R\|\leq\left(\frac{d(u)}{n}+\ep\right)|R|.\]

An \emph{$(\alpha, \ep, \varrho)$-strong reservoir} is an $(\ep,\varrho)$-weak
reservoir $R$ such that $G[R]$ has no $\alpha$-extreme sets.

An \emph{$(\alpha, \beta, \ep,\varrho)$-special reservoir} is an $(\ep, \varrho)$-weak
reservoir $R$ such that for all special sets $S\subseteq V(G)$, $S\cap R$ is not $(\alpha, \beta)$-extreme in $G[R]$.
\end{defn}
A routine application of Chernoff's bound yields $(\ep,\varrho)$-weak
reservoirs $R$ in moderately large graphs. The reason for this is
that we have only polynomially many conditions to preserve. A similar
observation allows us to construct $(\alpha, \beta, \ep,\varrho)$-special
reservoirs. However this standard approach fails for $(\alpha,\ep,\varrho)$-strong
reservoirs, because there are exponentially many conditions to check.

A connecting lemma should state that any two disjoint ordered edges in $V(G)\setminus R$
can be connected by a short square path whose interior vertices are in
$R$. For example, Fan and Kierstead \cite{FK1} proved:
\begin{lem}
If $\delta(G)>\frac{2}{3}|G|$ then there exists a square path connecting
any two disjoint edges.
\end{lem}
\noindent In the context of Theorem~\ref{thm:FK1}, $(\ep/2,\varrho)$-weak
reservoirs are sufficient since the degree bounds ensure that $\delta(G[R])>\frac{2}{3}|R|$.
In \cite{LSS,RRS2} the authors prove connecting lemmas for strong
reservoirs. We use a simpler argument and show that it works for special
reservoirs.

\subsection{Optimal paths}

Let $e_{1}:=v_1v_2$ and $e_{2}:=v_{s-1}v_s$ be disjoint ordered edges. A square $(e_{1},e_{2})$-\emph{path} is a square path of the form $v_1v_2\dots v_{s-1}v_s$.
\begin{defn}
An \emph{optimal} square path (or cycle, or $(e_{1},e_{2})$-path)
is a square path (or cycle, or $(e_{1},e_{2})$-path) $P$ such
that among all square paths (or cycles, or $(e_{1},e_{2})$-paths)
(i) $P$ is as long as possible, (ii) subject to (i), $P$ has as
many 3-chords as possible, and (iii) subject to (i) and (ii), $P$
has as many 4-chords as possible. 
\end{defn}
All the work in \cite{FK1,FK2,FK3} starts with lemmas about optimal
square paths. 

\begin{lem}
[Fan-Kierstead \cite{FK1}, \cite{FK2} Lemma 1]\label{segment} Suppose that
$P$ is a square path in a graph $G$ and $v\in V(G-P)$.
If $P$ is an $(e_{1},e_{2})$-optimal square path then $\|v,Q\|\leq\frac{2}{3}|V(Q)|+1$
for every segment $Q$ of $P$. Moreover, if $P$ is an optimal square
path then $\|v,P\|\leq\frac{2}{3}|P|-\frac{1}{3}$ and if $P$ is
an optimal square cycle then $ \|v,P\|\leq\frac{2}{3}|P|+\frac{1}{3}$.
\end{lem}

In the extremal case we will take advantage of the following fact.

\begin{cor}\label{3k}
P\'osa's Conjecture is true, if it holds for all $G$ with $|G|$ divisible by $3$. 
\end{cor}
\begin{proof}Suppose $|G|=3k+r$, where
$1\le r\le2$. Let $G'$ be $G$  with $r$ vertices deleted. Then $$\delta(G')\ge\lceil\frac{2}{3}(3k+r)\rceil-r=2k=\frac{2}{3}|G'|.$$
Thus by hypothesis, $G'$ has a hamiltonian square cycle $C'$. So an optimal
square cycle $C$ in $G$ has length at least $3k.$ Suppose $C$
is not hamiltonian in $G$. Then there exists $x\in V(G-C)$. By Lemma~\ref{segment},
we have the following contradiction: \[
2k+r\leq\delta(G)\leq\left\Vert v,C\right\Vert +|G|-|C|-1\le|G|-\frac{1}{3}|C|-\frac{2}{3}\leq 2k+r-\frac{2}{3}.\]
\end{proof}

We will also need:
\begin{lem}
[Fan-Kierstead \cite{FK2}, Lemma 9]\label{edgetopath} Let $P$
be an optimal square path of $G$. Let $xy$ be an edge of $G-P$
such that there are square paths, of at least $q$ vertices, starting
at $xy$ and $yx$ in $G-P$. If $|P|\geq 2q+2$, then $\|xy,P\|\leq\frac{4}{3}|P|-\frac{2}{3}q+2$. 
\end{lem}

\subsection{Probability}

If $X$ is a random variable with hypergeometric distribution (and 
our experiment consists of drawing $n$ items from a collection of
$N$ total items, $m$ of which are good and $N-m$ of which are bad)
the expected value of $X$ is given by \[
\mathbb{E}X=\sum_{k=0}^{n}k\cdot Pr(X=k)=\sum_{k=0}^{n}k\cdot\frac{\binom{m}{k}\binom{N-m}{n-k}}{\binom{N}{n}}=\frac{nm}{N}.\]

\begin{thm}
[Chernoff's bound \cite{Ch,JLR}]\label{Chernoff} Let $X$ be a random variable
with binomial or hypergeometric distribution. Then the following hold:\end{thm}
\begin{enumerate}
\item $Pr(X\geq\mathbb{E}X+t)\leq\exp\left(-\frac{t^{2}}{2(\mathbb{E}X+t/3)}\right),~~t\geq0$;
\item $Pr(X\leq\mathbb{E}X-t)\leq\exp\left(-\frac{t^{2}}{2\mathbb{E}X}\right),~~t\geq0$;
\item If $0<\g\leq3/2$, then $Pr(|X-\mathbb{E}X|\geq\g\mathbb{E}X)\leq2\exp\left(-\frac{\g^{2}}{3}\mathbb{E}X\right)$.
\end{enumerate}

\section{Non-extremal case\label{sec:nex}}

In this section we prove Theorem \ref{non-extremal}.  We have compromised optimality somewhat in our constructions 
 and calculations in favor of clarity of exposition.  For instance, we know how to reduce $n_0$ by a factor of $2$.  That being said, we can make the reservoir lemma slightly simpler and we can choose ``nicer'' constants throughout the non-extremal case at the cost of a factor of $3$ in $n_0$.

We first show that if $H$ is a graph with no $(\alpha,\beta)$-extreme special
sets whose minimum degree is almost $\frac{2}{3}|H|$, then any two
disjoint edges in $H$ can be connected by a short square path.  Let $xy\in E(H)$; we say that $P\{xy\}Q$ is a square path if one of $PxyQ$ or $PyxQ$ is a square path.

\begin{lem}
[Connecting Lemma] \label{connecting} Let $0<\beta<\alpha\leq\frac{1}{36}$,
$0<\ep\leq\frac{\alpha-\beta}{15.1}$, $l:=10$ and suppose $n\geq\max\{\frac{660}{\ep},\frac{69}{\beta}\}$.
Let $H=(V,E)$ be a graph on $n$ vertices with no $(\alpha,\beta)$-extreme
special sets such that $\delta(H)\geq(\frac{2}{3}-\ep)n$. Let $L\subseteq V$
such that $|L|\leq l$. If $ab$, $cd$ are any two disjoint
ordered edges in $H-L$, then there is a square $(ab, cd)$-path $P$
of order at most $14$ for which $V(P)\subseteq V\setminus L$. \end{lem}
\begin{proof}

Let $ab$, $cd$ be disjoint ordered edges in $H-L$ and set $A:=\{a,b,c,d\}$.
Here is our plan. First (a) we find disjoint edges $a'b',c'd'$
in $H-L-A$ such that $\|ab,a'b'\|=4=\|cd,c'd'\|$. Then, setting $A':=\{a',b',c',d'\}$,
(b) we construct a square path $\{a'b'\}Q\{c'd'\}$ with $Q\subseteq H':=H\setminus(L\cup A\cup A')$
connecting the unordered edges $a'b',c'd'$. This will yield a square
path $ab\{a'b'\}Q\{c'd'\}cd$, where the order of $\{a'b'\}$ and $\{c'd'\}$
is determined by $Q$.

Let $M\subseteq V$ with $|M|\le l+12$. We will often use the following
statement: \begin{equation}
\textrm{If \ensuremath{S} is a special set with \ensuremath{|S|\ge(1-\alpha+\beta)\frac{n}{3}} then \ensuremath{\|S\setminus M\|>0}.}\label{ne}\end{equation}
To see this, note that since $S$ is not $(\alpha,\beta)$-extreme and $n\geq \frac{69}{\beta}$, $S$ has at least $\floor{\beta\frac{n}{3}}>l+12$ vertices with degree at least $\alpha\frac{n}{3}>l+12$.

Consider the special set $N(a,b)=(N(a,a,a)\cup N(a,a,a))\cap N(b)$.
Since $\delta(H)\geq(\frac{2}{3}-\ep)n$, we have \[
|N(a,b)|\geq(1-6\ep)\frac{n}{3}\geq(1-\alpha+\beta)\frac{n}{3}.\]
 By \eqref{ne}, there exists $a'b'\in E(N(a,b)\setminus(L\cup A))$.
Likewise there is an edge $c'd'\in E(N(c,d)\setminus(\{a',b'\}\cup L\cup A))$,
completing the first goal (a).

Next we show (b). Let $V':=V(H')$. Then $|V'|\ge n-l-8$. We must
construct $Q\subseteq H'$. For $i\in[4]$, let $S_{i}:=S_i(A')=\{v\in V:\|v,A'\|=i\}$.
Then 
\begin{equation}
\frac{8}{3}n-4\ep n=4(\frac{2}{3}-\ep)n\le\|A',V\|  =\sum_{i\in[4]}i|S_{i}|\le4|S_{4}|+3|S_{3}|+2(n-|S_{4}|-|S_{3}|),\label{edges1}
\end{equation}
which gives
\begin{equation}
2|S_{4}|+|S_{3}|\geq \frac{2}{3}n-4\ep n.\label{edges2}
\end{equation}
\textbf{Case 1:} $|S_{4}|>l+12$. Looking ahead to an application
in Case 2.a, we will construct $Q\subseteq H'':=H'-A''$, for any fixed
$4$-set $A''$. Set $V'':=V(H'')$. By the case assumption, there exists $x\in S_{4}\cap V''$.
If there exists $u\in N(x)\cap(S_{4}\cup S_{3})\cap V''$ then set
$Q:=\{xu\}$. Otherwise, $|S_{4}|+|S_{3}|\le\frac{1}{3}n+\ep n+l+12$,
since $d(x)\ge\frac{2}{3}n-\ep n$. Thus by \eqref{edges2}, and using
$\alpha-\beta\ge15.1\ep$ and $n\geq \frac{660}{\ep}$, we have \[
|S_{4}|\ge\frac{1}{3}n-5\ep n-l-12\ge(1-\alpha+\beta)\frac{n}{3}.\]
 Moreover, $S_{4}=N(a',b',c',d')=(N(a',b',c')\cup N(a',b',c'))\cap N(d')$
is special. Thus by \eqref{ne}, there exists an edge $uv\in S_{4}\cap V''$,
and we set $Q:=uv$.

\noindent \textbf{Case 2:} $|S_{4}|\le l+12$.  Let 
\begin{align}
T_{1} & :=\{v\in S_{3}\cup S_4:\|v,\{a',b'\}\|=2\}=(N(a',b',c')\cup N(a',b',d'))\cap N(a')\textrm{ and }\\
T_{2} & :=\{v\in S_{3}\cup S_4:\|v,\{c',d'\}\|=2\}=(N(c',d',a')\cup N(c',d',b'))\cap N(c').
\end{align}
Then $T_1$ and $T_2$ are both special sets.  Note that $S_3$ is partitioned as $(T_1\setminus S_4)\cup (T_2\setminus S_4)$ and $T_1\cap T_2=S_4$.  By \eqref{edges2} and the fact that $|T_1|+|T_2|=|S_3|+2|S_4|$, we have
\begin{equation}
|T_1|+|T_2|\geq \frac{2}{3}n-4\ep n \label{T12}.
\end{equation}
Without loss of generality, $|T_{1}|\leq|T_{2}|$, and so $T_{2}\neq\emptyset$.  Finally, note that by \eqref{edges2} and the case assumption we have,
\begin{equation}
|T_1\cup T_2|=|S_{3}\cup S_4|\geq\frac{2}{3}n-4\ep n-l-12.\label{S3L}\end{equation}

\noindent \textbf{Case 2.a:} $|T_{1}|>l+8$. If there exists $xy\in E(T_{1},T_{2})\cap E(H')$,
then set $Q:=xy$. Otherwise, let $x\in T_{1}\cap V'$. Then using,
in order, $d(x)\geq(\frac{2}{3}-\ep)n$, \eqref{T12}, $\alpha-\beta\ge15.1\ep$ and $n\geq \frac{660}{\ep}$
we have \begin{equation}
\frac{n}{3}+\ep n+l+8\ge|T_{2}|\ge|T_{1}|\ge\frac{n}{3}-5\ep n-l-8\ge(1-\alpha+\beta)\frac{n}{3}.\label{T2}\end{equation}

By \eqref{ne} and \eqref{T2}, there exist edges $a''b''\in E(T_{1})$
and $c''d''\in E(T_{2})$ such that $A'':=\{a'',b'',c'',d''\}$ is disjoint from $L\cup A\cup A'$.  Note that $A''\cap S_4=\emptyset$, since $E(T_{1},T_{2})\cap E(H')=\emptyset$ as mentioned above.

Set $U:=V\setminus (T_1\cup T_2)$. By \eqref{S3L}, \begin{equation}
|U|=n-|T_1\cup T_2|\le\frac{n}{3}+4\ep n+l+12.\label{U}\end{equation}

By \eqref{T2}, for any $x\in A''$, \begin{equation}
\|x,U\|\ge\frac{2}{3}n-\ep n-|T_{2}|\ge\frac{n}{3}-2\ep n-l-8.\label{dU}\end{equation}
 By \eqref{U}, \eqref{dU}, and $n\geq\frac{660}{\ep}$, we have $\overline{\|x,U\|}\le6\ep n+3l+32<\frac{1}{5}|U\cap V''|$.
Thus there exist more than $l+12$ vertices in $S_4(A'')$.
Thus by Case 1, there exists a square path $Q:=\{a''b''\}Q'\{c''d''\}$ with $|Q'|\le 2$.

\noindent \textbf{Case 2.b:} $|T_{1}|\le l+8$. Then $|T_{2}|\geq\frac{2}{3}n-4\varepsilon n-l-8$ by \eqref{T12}.
Let $x\in N(a',b')\cap V'$, and note that $S:=T_{2}\cap N(x)=(N(a',c',d')\cup N(b',c',d'))\cap N(x)$
is a special set. Moreover by $\alpha-\beta\ge15.1\ep$ and $n\geq \frac{660}{\ep}$ we have\[
|S|\ge|T_{2}|+|N(x)|-n\ge\frac{n}{3}-5\ep n-l-8\ge(1-\alpha+\beta)\frac{n}{3}.\]
Thus by \eqref{ne}, there exists an edge $yz\in E(S\cap V')$. Let
$Q:=xyz$.

\noindent 
\end{proof}

Now we prove the reservoir lemma.

\begin{lem}[Reservoir Lemma] \label{reservoir} 
Let $\alpha\geq \frac{1}{36}$, $c\geq\frac{1}{14}$,  $\alpha':=(1-3c)\alpha$, $\beta':=c\alpha$, $\ep\geq \frac{\alpha'-\beta'}{15.1}$, $\varrho\geq 
1-\frac{2/3+\ep}{5/6-2\ep}$ and $n\geq n_{0}:=\n$. If $H$ is a
graph on $n$ vertices such that $\delta(H)\geq \frac{2}{3}n$ and $H$ contains 
no $\alpha$-extreme sets, then $H$ contains an $(\alpha', \beta', \ep, \varrho)$-special
reservoir. \end{lem}

\begin{proof}
Let $\gamma:=\frac{2\beta'}{1-\alpha'-\beta'}$. We will show that there
exists a set $R\subseteq V(H)$ such that $|R|=\ceiling{\varrho n}$ which satisfies the following three
properties.
\begin{enumerate}
\item[(i)] For all $u\in V(H)$, $\left(\frac{d(u)}{n}-\ep\right)|R|\leq\|u,R\|\leq\left(\frac{d(u)}{n}+\ep\right)|R|$.
\item[(ii)] For all special sets $S\subseteq V(H)$, if $|S|\geq(1-\alpha'+\beta')\varrho\frac{ n}{3}$, then 
$|S\cap R|\leq 1.05\varrho|S|$ and for all special sets $S\subseteq V(H)$, if $|S\cap R|\geq(1-\alpha'+\beta')\varrho\frac{ n}{3}$, then $|S\cap R|\leq (1+\gamma)\varrho|S|$.
\item[(iii)] For all special sets $S\subseteq V(H)$, if $|S|\geq(1-\alpha'-\beta')\frac{n}{3}$, then 
there exists a set $T'\subseteq S\cap R$ such that $|T'|\geq \beta'\varrho\frac{n}{3}$
and $\|z,S\cap R\|\geq \alpha'\varrho\frac{n}{3}$ for all $z\in T'$. \end{enumerate}
Then we will show that these three properties imply that $R$ is an 
$(\alpha', \beta', \ep, \varrho)$-special reservoir.

Let $R\subseteq V(H)$ be a set of size $\ceiling{\varrho n}=:r$ chosen
at random from all $\binom{n}{r}$ possibilities. There are five calculations that follow.  In each of these calculations we will need $n$ to be large, specifically $n\geq \n$ is large enough.

Let $u\in V(H)$. The expected value of $\|u,R\|$ is $\frac{rd(u)}{n}\geq\varrho d(u)$.
So by Theorem \ref{Chernoff}.3, we have \begin{align*}
Pr\left(\left|\|u,R\|-\frac{rd(u)}{n}\right|\geq\frac{\ep n}{d(u)}\frac{rd(u)}{n}\right)\leq2\exp\left(-\frac{(\frac{\ep n}{d(u)})^{2}}{3}\frac{rd(u)}{n}\right)\leq2\exp\left(\frac{-\ep^{2}\varrho n^2}{3d(u)}\right)<\frac{1}{3n}.\end{align*}
There are $n$ vertices in $V(H)$. So by applying Boole's inequality, the probability that there exists a vertex which does not satisfy property (i) is less than $1/3$.

Let $S\subseteq V(H)$ be a special set such that $|S|\geq(1-\alpha'+\beta')\varrho\frac{ n}{3}$. 
The expected value of $|S\cap R|$ is $\frac{r|S|}{n}\geq\varrho|S|\geq (1-\alpha'+\beta')\varrho^2\frac{ n}{3}$.
So by Theorem \ref{Chernoff}.1, we have
\begin{align*}
\log Pr(|S\cap R|\geq1.05\frac{r|S|}{n})\leq-\frac{(.05\varrho|S|)^{2}}{2(\varrho|S|+.05\varrho|S|/3)}\leq-\frac{.0025\varrho^{2}(1-\alpha'+\beta')}{2(1+.05/3)}\frac{n}{3}<\log\frac{1}{9n^{5}}.\end{align*}
So with high probability, \begin{equation}\label{WHP}|S\cap R|\leq 1.05\varrho |S|  \text{ for all } S\subseteq V(H) \text{ such that } |S|\geq(1-\alpha'+\beta')\varrho\frac{ n}{3} .\end{equation}  Now let $S\subseteq V(H)$ be a special set such that $|S\cap R|\geq(1-\alpha'+\beta')\varrho\frac{ n}{3}$.  Since $|S|\geq |S\cap R|$ we have $|S|\geq (1-\alpha'+\beta')\varrho\frac{ n}{3}$ and thus by \eqref{WHP}, $|S|\geq \frac{|S\cap R|}{1.05\varrho}\geq \frac{(1-\alpha'+\beta')}{1.05}\frac{n}{3}$. The expected value of $|S\cap R|$ is $\frac{r|S|}{n}\geq\varrho|S|\geq \varrho \frac{(1-\alpha'+\beta')}{1.05}\frac{n}{3}$.  Using Theorem~\ref{Chernoff}.1 again, we have
\begin{align*}
\log Pr(|S\cap R|\geq(1+\gamma)\frac{r|S|}{n})\leq-\frac{(\gamma\varrho|S|)^2}{2(\varrho|S|+\gamma\varrho|S|/3)}\leq-\frac{\gamma^{2}\varrho(1-\alpha'+\beta')}{1.05(2+2\gamma/3)}\frac{n}{3}<\log\frac{1}{3n^{5}}.\end{align*}
There are at most $n^5$ special sets $S\subseteq V(H)$. So by applying Boole's inequality, the probability that there exists a set $S$ which does not satisfy property (ii) is less than $4/9$.

Let $S\subseteq V(H)$ be a special set such that $|S|\geq(1-\alpha'-\beta')\frac{n}{3}=(1-\alpha+2c\alpha)\frac{n}{3}$.
Since $H$ has no $\alpha$-extreme sets, we see by Lemma \ref{nicevertices} that $S$ is not $(\alpha, 2c\alpha)$-extreme. So there exists a set $S'\subseteq S$ having
the property that $|S'|=\floor{2c\alpha\frac{n}{3}}$ and for all $v\in S'$,
$\|v,S\|\geq\alpha\frac{ n}{3}$. Let $T':=S'\cap R$.
We first show that with high probability, $|T'|\geq\frac{3\varrho}{4}|S'|\geq \frac{\varrho}{2}(|S'|+1)\geq \beta'\varrho\frac{n}{3}$.
The expected value of $|T'|$ is $\varrho|S'|\geq\varrho(2c\alpha\frac{n}{3}-1)$.  So by Theorem \ref{Chernoff}.2, we have \begin{align*}
\log Pr(|T'|\leq\varrho|S'|-\frac{\varrho}{4}|S'|)\leq-\frac{(\frac{\varrho}{4}|S'|)^{2}}{2(\varrho|S'|)}=-\frac{\varrho|S'|}{32} \leq-\frac{\varrho(2c\alpha\frac{n}{3}-1)}{32}<\log\frac{1}{9n^{5}}.\end{align*}

Next we show that, with high probability, every vertex in $S'$ has
at least $(1-3c)\varrho \|v,S\|\geq \alpha'\varrho\frac{n}{3}$ neighbors in $S\cap R$.
Let $v\in S'$. The expected value of $\|v,T\|$ is $\varrho\|v,S\|\geq \varrho\alpha\frac{n}{3}$. So by Theorem \ref{Chernoff}.2, we have
\begin{align*}
\log Pr(\|v,S\cap R\|\leq(1-3c)\varrho\|v,S\|)\leq-\frac{(3c\varrho\|v,S\|)^2}{2\varrho\|v,S\|}=-\frac{9c^2\varrho \|v,S\|}{2}\leq -\frac{3c^2\varrho \alpha n}{2}<\log \frac{1}{9n^{6}}.\end{align*}
There are at most $n^5$ special sets $S\subseteq V(H)$ and at most $n^6$ sets defined when we examine the neighborhood of vertices in each special set.  So by applying Boole's inequality, the probability that there exists a set $S$ which does not satisfy property (iii) is less than $2/9$.

The probability that $R$ doesn't satisfy one of the conditions is less than $1$, thus there exists a set $R\subseteq V(H)$ satisfying properties (i)-(iii).

We now show that $R$ is an $(\alpha', \beta', \ep, \varrho)$-special reservoir.  Since $R$ satisfies property (i), $R$ is a $(\ep, \varrho)$-weak reservoir.  Let $S\subseteq V(H)$ be a special set such that $|S\cap R|\geq(1-\alpha'+\beta')\varrho\frac{ n}{3}$.
By property (ii), we have $\varrho|S|(1+\gamma)\geq|S\cap R|\geq(1-\alpha'+\beta')\varrho\frac{ n}{3}$, and thus $$|S|\geq\frac{(1-\alpha'+\beta')}{1+\gamma}\frac{n}{3}=(1-\alpha'-\beta')\frac{n}{3}.$$
Then since $|S|\geq(1-\alpha'-\beta')\frac{n}{3}$ 
there is, by property (iii), a set of vertices $T'\subseteq S\cap R$ with $|T'|\geq\beta'\varrho\frac{ n}{3}$
such that for all $v\in T'$, $\|v,S\cap R\|\geq\alpha'\varrho\frac{ n}{3}$.
Thus $S\cap R$ is not $(\alpha', \beta')$-extreme in $G[R]$.  Therefore $R$ is an $(\alpha', \beta', \ep, \varrho)$-special reservoir.
\end{proof}

We now prove a lemma which allows us to cover most of the complement
of the reservoir with at most two long square paths.

\begin{lem}[Path Cover Lemma]
\label{pathcover} Suppose  
$\ep\le\frac{1}{500}$ and $n\ge6000$. Let $H$ be a graph on $n$ vertices with $\delta(H)\geq\left(\frac{2}{3}-\ep\right)n$.  
Then 

\textup{(a)} $H$ has a square path $P$ with $|P|\geq(\frac{1}{2}-3\ep)n$. 
\label{longpath}

\textup{(b)} $H$ has two vertex disjoint square paths $P_{1}$ and $P_{2}$ so that $|P_{1}|+|P_{2}|>(\frac{5}{6}-2\ep)n$. \label{twopaths} \end{lem}
\begin{proof}
(a) Let $P:=u_{1}u_{2}...u_{p}$ be an optimal square path in $H$ and suppose that $p<(\frac{1}{2}-3\ep)n$.  We first observe that since $\delta(H)\geq (\frac{2}{3}-\ep)n$ we have $N(u_1,u_2)\geq (\frac{1}{3}-2\ep)n$ and thus $p>(\frac{1}{3}-2\ep)n$.
Let $H':=H-P$ and set $h:=|H'|$. If $\|v,P\|\leq (\frac{2}{3}-4\ep)p$ for all $v\in V(H')$ then we have $\delta(H')\geq(\frac{2}{3}-\ep)n-(\frac{2}{3}-4\ep) p \geq\frac{2}{3}h$.
Thus by Theorem~\ref{thm:FK2},  $H'$ has a hamiltonian square path of length more than than $\frac{1}{2}n$, contradicting the optimality of $P$. Thus there is a vertex $x\in V(H')$ such that $\|x,P\|>(\frac{2}{3}-4\ep)p>\frac{1}{2}p+1$. It follows that $x$ is adjacent to two consecutive vertices of $P$. Choose $i\in[p]$ as small as possible such that $u_{i},u_{i+1}\in N(x)$. Let $Q:=u_{1}u_{2}...u_{i-1}$ and set $q:=i-1$. Then $\|x,Q\|\leq\frac{1}{2}q$. We claim that $q<(\frac{1}{6}-2\ep)n$.
Otherwise, \begin{align*}
\|x,P-Q\|>(\frac{2}{3}-4\ep)p-\frac{1}{2}q & =\frac{2}{3}(p-q)+\frac{1}{6}q-4\ep p\\
 & >\frac{2}{3}|P-Q|+\frac{1}{6}(\frac{1}{6}-2\ep)n-4\ep(\frac{1}{2}-3\ep)n\\
 & >\frac{2}{3}|P-Q|+\frac{1}{36}n-\frac{7}{3}\ep n\\
 & >\frac{2}{3}|P-Q|+1,  \end{align*}
contradicting Lemma \ref{segment}. On the other hand, since $|N(x,u_{i})|\geq(\frac{1}{3} -2\ep)n=\frac{2}{3} (\frac{1}{2}-3\ep)n>\frac{2}{3}p$, Lemma~\ref{segment} implies $x$ and $u_{i}$ have a common neighbor $y$ in $H'$. Also, by Lemma \ref{segment} we have \[
\delta(H')\geq(\frac{2}{3}-\ep)n-(\frac{2}{3}p-\frac{1}{3})>\frac{2}{3}h-\ep n,\] and thus for any edge $uv$ in $H'$, $|N_{H'}(u,v)|\geq\frac{1}{3}h-2\ep n>(\frac{1}{6}-2\ep)n$. Hence, we can find a square path $P'$ of length at least $(\frac{1}{6}-2\ep)n$ starting at $xy$. Since $|P'|>q$, the square path $P'yxu_{i}u_{i+1}...u_{p}$ is longer than $P$, a contradiction. This completes the proof of part (a).

(b) Let $P_{1}$ be an optimal square path in $H$ and let $p:=|P_{1}|$. Note that $p\geq(\frac{1}{2}-3\ep)n$ by Lemma \ref{pathcover}(a).  If $p>(\frac{5}{6}-2\ep)n$, then set $P_2=\emptyset$ and we are done. So we may assume that $p\leq (\frac{5}{6}-2\ep)n$. Set $H':=H-P_{1}$ and $h:=|H'|>n/6$. If $\|v,P_{1}\|\leq(\frac{2}{3}-3\ep)p$ for all $v\in V(H')$
then $\delta(H')\geq(\frac{2}{3}-\ep)n-(\frac{2}{3}-3\ep) p\geq\frac{2}{3}h$. Thus  $H'$ has a hamiltonian square path $P_{2}$ by Theorem~\ref{thm:FK2}, and we are done. Otherwise, let $x\in V(H')$ such that $\|x,P_{1}\| >(\frac{2}{3}-3\ep)p$. Note that by Lemma \ref{segment}, we have $\delta(H')\geq(\frac{2}{3}-\ep)n -(\frac{2}{3}p -\frac{1}{3}) > \frac{2}{3}h-\ep n$, and thus there is a square path of length at least $\frac{1}{3}h-2\ep n$ starting at any ordered edge in $H'$. Set $H'':=G[N_{H'}(x)]$ and $h':=|H''|$. Note that by Lemma \ref{edgetopath}, we have that for all
$y\in V(H'')$, 
\begin{align*}
\|y,P_1\|<\frac{4}{3}p-\frac{2}{3}(\frac{1}{3}h-2\ep n)+2-(\frac{2}{3}-3\ep)p=\frac{2}{3}p-\frac{2}{9}h+\frac{4}{3}\ep n+3\ep p+2,\end{align*}
so \begin{align*}
\|y,H'\|>(\frac{2}{3}-\ep)n-(\frac{2}{3}p-\frac{2}{9}h+\frac{4}{3}\ep n+3\ep p+2)=\frac{8}{9}h-\frac{7}{3}\ep n-3\ep p-2.\end{align*}
So every vertex in $H''$ has at most $\frac{1}{9}h+\frac{7}{3}\ep n+3\ep p+1$ nonneighbors in $H'$. Therefore \begin{align*}
\delta(H'')\geq\frac{\frac{2}{3}h-\ep n-(\frac{1}{9}h+\frac{7}{3}\ep n+3\ep p+1)}{\frac{2}{3}h-\ep n}h' >\frac{2}{3}h',\end{align*}  
 since $\ep\leq\frac{1}{500}$, $n\ge6000$, and $h>n/6$. Therefore $H''$ has a hamiltonian square path $P_{2}$. Thus \begin{align*}
|P_{1}|+|P_{2}| > p+\frac{2}{3}h-\ep n = n-\frac{1}{3}h-\ep n \geq n-\frac{1}{3}(\frac{1}{2}+3\ep)n-\ep n = (\frac{5}{6}-2\ep)n.\end{align*}

\end{proof}

Now we are ready to finish the nonextreme case.
\begin{proof}[Proof of Theorem \ref{non-extremal}]
Let $\alpha:=\frac{1}{36}$ and let $G$ be a graph on $n$ vertices. Suppose $G$ has no $\alpha$-extreme sets, $n\geq n_{0}:=\n$, and $\delta(G)\geq \frac{2}{3}n$.  Let $c:=\frac{1}{14}$, $\ep:=\frac{50}{1057}\alpha$, and $\varrho:=1-\frac{2/3+\ep}{5/6-2\ep}$.  Apply Lemma~\ref{reservoir} to obtain an $(\frac{11}{14}\alpha, \frac{1}{14}\alpha, \ep, \varrho)$-special reservoir $R$. Let $H:=G-R$ and let $h:=|H|$. Since $R$ is a special
reservoir we have $\delta(H)\geq(\frac{2}{3}-\ep)h$. Now we apply
Lemma~\ref{pathcover} to $H$, to get disjoint square paths $P_{1}$
and $P_{2}$ so that $$|P_{1}|+|P_{2}|>(\frac{5}{6}-2\ep)h=(\frac{5}{6}-2\ep)(n-\ceiling{\varrho n})\geq (\frac{2}{3}+\ep)n-1>\frac{2}{3}n.$$
Since $R$ is a special reservoir, every special set $S\subseteq V(G)$ has the property that $S\cap R$ is not $(\frac{11}{14}\alpha, \frac{1}{14}\alpha)$-extreme in $G[R]$. So we apply Lemma \ref{connecting} at most twice to connect the paths $P_{1}$ and $P_{2}$ through $R$.
On the second application, we set $L:=V(P_1)\cap R$ to make sure that we avoid the vertices used
in the first application. This gives us a square cycle $C$ with $V(P_{1})\cup V(P_{2})\subseteq V(C)$
and thus $|C|>\frac{2}{3}n$. Therefore $G$ has a hamiltonian square cycle by Theorem \ref{thm:FK3}.
\end{proof}

\section{Extremal Case\label{sec:Ex}}

In this section we prove Theorem \ref{thm:good}. First we need two
propositions. Note that the length of an (ordinary) path $P$ is the size $\|P\|$ of its edge set.
\begin{prop}
\label{pro:LPath}Every connected graph $H$ with $|H|\geq 3$ has a path or cycle of
length $\min(2\delta(H),|H|)$. 
\end{prop}
\begin{proof}
Let $P$ be a maximum length path in $H$.  If we are not done, then $\|P\| < 2\delta(H)$. So, as in the proof of Dirac's Theorem \cite{D}, 
$G$ has a cycle $C$ that spans $V(P)$. If $C$ is hamiltonian then we are done; otherwise, using connectivity, we can extend $C$ to a path longer than $P$, a contradiction.
\end{proof}
\begin{prop}
\label{pro:LCyc}If $H$ is a graph with circumference $l>|H|-\delta(H)$,
then $l\geq\min(2\delta(H),|H|)$, and moreover, if $|H|$
is also even, then $H$ has an even cycle of length at least $\min(2\delta(H),|H|)$.\end{prop}
\begin{proof}
Let $C\subseteq H$ be a cycle of length $l$, and
fix an orientation of $C$. If $|C|=|H|$ then we are done, even if
$|H|$ is even. Otherwise, let $P:=v_{1}\dots v_{p}$ be a maximum
path in $H-C$. Then all neighbors of $v_{p}$ are on $P\cup C$.
By hypothesis $\delta(H)>|H|-l\geq p$, and so $v_{1}$ has a neighbor
$x\in C$ and $v_{p}$ has a neighbor on $C-x$. Let $y,z\neq x$ be neighbors
of $v_{p}$ on $C$ with $y$ as close as possible to $x$ in the
forward direction and $z$ as close as possible in the backward direction
(possibly $y=z$). Then $\left\Vert zCx\right\Vert ,\left\Vert xCy\right\Vert \geq p+1$,
as otherwise we could replace the interior vertices of one of these segments
with $P$ to obtain a longer cycle, which would yield a contradiction. Moreover, since $C$ has maximum length, any two neighbors of $v_{p}$ are separated by
at least one vertex on $C$. Since $v_{p}$ has at least $d(v_{p})-p$
neighbors on $C-x$, \[
|C|=\left\Vert xCy\right\Vert +\left\Vert yCz\right\Vert +\left\Vert zCx\right\Vert \geq(p+1)+2(d(v_{p})-p-1)+(p+1)\geq2\delta(H).\]

Now suppose $|H|$ is even. If $|C|$ is even we are done, so suppose
$|C|$ is odd. Consider the path $P$ and vertices $x,y,z$ defined
above. If $\|xCy\|$ and $\|zCx\|$ have different parity, then replace
$xCy$ with $xPy$ or replace $zCx$ with $zPx$ to get an even cycle
of length at least $2\delta(H)$. So assume $\|xCy\|$ and $\|zCx\|$
have the same parity, and thus $\|yCz\|$ is odd. Now $v_{p}$ has
$k\geq d(v_{p})-p$ neighbors on $yCz$. Let $y=a_{1},a_{2},\dots,a_{k}=z$
be the neighbors of $v_{p}$ on $yCz$ in their natural order. Since
$\|yCz\|$ is odd, some segment $a_{i}Ca_{i+1}$ must have odd length.
By replacing $a_{i}Ca_{i+1}$ with $a_{i}v_{p}a_{i+1}$, we get a
cycle $C'$ with even length such that $|C'|\geq(p+1)+(p+1)+2(d(v_{p})-p-1)\geq2\delta(H)$ as before. 
\end{proof}

\begin{proof}[Proof of Theorem~\ref{thm:good}]
Let $G=(V,E)$ be a graph on $n$ vertices with $\delta(G)\geq \frac{2}{3}n$.  By Corollary \ref{3k} we may assume $n=3k$, which gives $\delta(G)\geq 2k$. Set $\alpha:=\frac{1}{36}$, and suppose $G$ has an $\alpha$-extreme
subset.  Let $S\subseteq V$ be an $\alpha$-extreme set of minimal order, so $|S|= \lceil (1-\alpha)k\rceil$. Set $T:=V\setminus S$. If $k<1/\alpha$, then $|S|=k$, $|T|=2k$, $G[S,T]$ is complete and $\delta(G[T])\geq k$.  So by Dirac's theorem $T$ has a hamiltonian cycle $C:=y_1\dots y_{2k}y_1$. Since $G[S,T]$ is complete we can insert the vertices $x_1, x_2, \dots, x_k$ of $S$ into $C$ so that $y_1y_2x_1y_3y_4x_2\dots y_{2k-1}y_{2k}x_ky_1y_2$ is a hamiltonian square cycle.  So for the rest of the proof assume $k\geq 1/\alpha$.  Choose $T_{0}\subseteq T$ such that $|V\setminus (S\cup T_0)|$ is even, 
$2\floor{\sqrt{\alpha}k}-1\leq|T_0|\le 2\floor{\sqrt{\alpha}k}$, and subject to this, $\left\Vert T_{0},S\right\Vert $
is as small as possible. Set $T_{1}:=T\setminus T_{0}$, and note that $|T_1|$ is even. We have, 
\begin{equation}
\forall x\in S,~\overline{\left\Vert x,T\right\Vert }\le k-(|S|-\left\Vert x,S\right\Vert) \leq 2\floor{\alpha k}.\label{eq:degx}\end{equation}
Every vertex in $T_1$ has at most as many nonneighbors in $S$ as every vertex in $T_0$. Thus, using $\alpha=\frac{1}{36}$, and expressing $k$ as $k=36q+r$ with $q,r\in \mathbb{Z}$ and $0\le r \le 35$, we have 
\begin{equation}
\forall y\in T_{1},~\overline{\left\Vert y,S\right\Vert }\leq \floor{\frac{2\floor{\alpha k}|S|}{|T_{0}\cup \{y\}|}}\leq \floor{\frac{2\floor{\alpha k}(k-\floor{\alpha k})}{2\floor{\sqrt{\alpha}k}}} \le \floor{\frac{(35q+r)}{6}}{\leq} \floor{\sqrt{\alpha}k}.
\label{eq:degy}
\end{equation}

Set $m:=k-|T_0|+\floor{\alpha k}$ and note that since $k\geq 36$,
\begin{equation}\label{eq:m}
m\geq \frac{2}{3}k+\floor{\alpha k}\geq \frac{2}{3}k+1. 
\end{equation}
Thus we have \begin{equation}
\delta(G[T_{1}])\geq 2k-|S\cup T_{0}|=k-|T_0|+\floor{\alpha k}=m\ge\frac{2}{3}k+1.\label{eq:T_1}
\end{equation}

\noindent 
\textbf{Case 1: }There exists an even cycle \textbf{$C\subseteq G[T_{1}]$}
of length $2l\geq2m$; say $C:=y_{1}\dots y_{2l}y_{1}$.  Looking ahead to an application in Case 2, we prove something slightly more general than what is needed for Case 1.  For some $t\leq |T_1|/2$, let $T_1'\subseteq T_1$ such that $|T_1'|=2t$.  Enumerate the vertices of $T_1'$ as $z_1,\dots,z_{2t}$.  Let $P:=\{p_{1},\dots,p_{t}\}$ be a set of \emph{ports}, where $p_{i}:=\{z_{2i-1},z_{2i},z_{2i+1},z_{2i+2}\}$ and addition of indices is modulo $t$.  We say that a vertex $x\in S$ can be inserted into port $p_{i}$ if $p_{i}\subseteq N(x)$.

\begin{claim}\label{ports}
For $S'\subset S$ with $|S'|\ge |S|-4$, let $\Gamma$ be the $S',P$-bigraph with $xp\in E(\Gamma)$ if and only if $x$ can be inserted into $p$. Then $\Gamma$ has a matching $M:=\{x_{i}p_{i}:i\in[t]\}$ that saturates $P$.
\end{claim}

\begin{proof}

Using Hall's Theorem, since $|S'|\geq |T_1|/2\geq |P|$, it suffices to show that  
\begin{equation}
\|x,P\|_{\Gamma}+\|S',p\|_{\Gamma}\ge|P|\textrm{ for all \ensuremath{x\in S' } and \ensuremath{p\in P}.}\label{H}\end{equation} 

If $x\in S'$, then $\overline{\|x,T\|}_{G}\leq 2\floor{\alpha k}$ by \eqref{eq:degx}.
Since each $y\in T_1'$ is in two ports, each nonedge $xy$ contributes
to two nonedges in $\Gamma$. So $\overline{\|x,P\|}_{\Gamma}\leq 4\floor{\alpha k}$.
Thus \begin{equation}
\|x,P\|_{\Gamma}\geq|P|-\overline{\|x,P\|}_{\Gamma}\geq |P|-4\alpha k.\label{SP}\end{equation}
 If $p\in P$, then $\overline{\|S',y\|}_{G}\leq \floor{\sqrt{\alpha}k}$ for each
$y\in p$ by \eqref{eq:degy}. Thus $\overline{\|S',p\|}_{\Gamma}\leq 4\floor{\sqrt{\alpha}k}$. 
So\begin{equation}
\|S',p\|_{\Gamma}\geq|S'|-\overline{\|S',p\|}_{\Gamma}\ge (1-\alpha-\frac{4}{k}-4\sqrt{\alpha})k.\label{PS}\end{equation}
Since $4\sqrt{\alpha}+5\alpha+\frac{4}{k}\leq \frac{33}{36}< 1$, summing \eqref{SP} and \eqref{PS} 
yields \eqref{H}.
\end{proof}

Let $S':=S$ and $P:=\{p_{1},\dots,p_{l}\}$, where $p_{i}:=\{y_{2i-1},y_{2i},y_{2i+1},y_{2i+2}\}$ and addition of indices is modulo $2l$. 
By Claim \ref{ports}, there exist $x_{1}, \dots, x_l$ such that $y_{1}y_{2}x_{1}y_{3}y_{4}x_{2}\dots y_{2l-1}y_{2l}x_{l}y_{1}y_{2}$
is a square cycle of length $3l$. By (\ref{eq:T_1}), $3l\geq 3m>2k$, and so Theorem~\ref{thm:FK3} implies that $G$ has a hamiltonian square
cycle.

\noindent 
\textbf{Case 2:} Not Case 1. Since $|T_1|$ is even, using Proposition~\ref{pro:LCyc} and \eqref{eq:T_1}, 
 \begin{equation}
|D|\leq|T_{1}|-\delta(G[T_{1}])\leq k,\text{ for every cycle }D\subseteq G[T_{1}].\label{eq:con}\end{equation}
First suppose $G[T_{1}]$ is connected.  By Proposition \ref{pro:LPath}, there exists a path in $G[T_1]$ of length at least $2m$.
\begin{claim}\label{nocross}
Let $P=y_1\dots y_l$ be a path of maximum length in $G[T_1]$.  If $y_i\in N(y_1)$ and $y_j\in N(y_l)$, then $i\leq j$.
\end{claim}

\begin{proof}
Suppose there exists $y_i\in N(y_1)$, $y_j\in N(y_l)$ such that $i>j$.  With respect to this condition, choose $y_i$ and $y_j$ such that $i-j$ is minimum.  If $i-j-1\leq \frac{1}{3}k$, 
set $D:=y_1\dots y_jy_l\dots y_iy_1$.  By \eqref{eq:m}, 
$|D|\geq 2m-\frac{1}{3}k>k$, which contradicts \eqref{eq:con}.  If $i-j-1>\frac{1}{3}k$, let $h$ be maximum such that $y_h\in N(y_1)$ and set $D:= y_1y_2\dots y_hy_1$.  Since $i-j-1>\frac{1}{3}k$ and $i-j$ is minimum, we have $|D|\geq h\geq m+i-j-1>k$, which contradicts \eqref{eq:con}.
\end{proof}

Let $P:=y_1\dots y_l$ be a path of maximum length in $G[T_1]$ and with respect to this condition, choose $P$ so that  $j-i$ is minimum, where $y_j$ is the smallest indexed neighbor of $y_l$ and $y_i$ the largest indexed neighbor  of $y_1$.  
Note that by Claim \ref{nocross}, $j-i\geq 0$. By \eqref{eq:con} we have,
\begin{equation}
 N(y_{1})\subseteq \{y_2,\dots y_k\} \textrm{ and } N(y_{l})\subseteq \{y_{l-k+1},\dots,y_{l-1}\}. \label{dis}\end{equation}

Set \[A:=\{y_{1},\dots,y_{i-1}\},\ B:=\{y_{i},\dots,y_{j}\},\ C:=\{y_{j+1},\dots,y_{l}\}.\]
Without loss of generality we may suppose $|A|\geq |C|$ and thus we have \begin{equation}\label{ACbounds}m\leq\delta(G[T_{1}])\leq|C|\leq |A|< k\end{equation} and $|B|= j-i+1\leq l-2m$.

Next we show that \begin{equation}
\left\Vert A,C\right\Vert =0.\label{dj}\end{equation}
 Suppose $a<i\leq j<b$ and $y_{a}y_{b}\in E$. Choose $y_{a'}\in N(y_{1})$
and $y_{b'}\in N(y_{l})$ such that $a<a'\le i\leq j\le b'<b$ and
both $a'-a$ and $b-b'$ are minimal.  Now $D:=y_{1}Py_{a}y_{b}Py_{l}y_{b'}Py_{a'}y_{1}$ is a cycle having the property that $N(y_1)\cup N(y_l)\subseteq V(D)$ and thus $|D|\geq|N(y_1)\cup N(y_l)|\geq 2m-1>k$, contradicting (\ref{eq:con}).

Set $A':=\{y_h\in A:y_{h+1}\in N(v_1)\}$ and $C':=\{y_h\in C: y_{h-1}\in N(y_l)\}$.  Note that $|A'|\geq m$ and $|C'|\geq m$.  
We claim that the vertices in $A'\cup C'$ are good in the sense that
\begin{equation}\label{A'C'}
\forall a\in A', N(a)\cap (T_1\setminus (A\cup \{y_i\}))= \emptyset \text{ and }  \forall c\in C', N(c)\cap (T_1\setminus (C\cup \{y_j\}))= \emptyset.                                                                                                                                                                                                                                                                                                                                                                                            \end{equation}
Without loss of generality, suppose some $y_h\in A'$ has a neighbor $y'\in T_1\setminus (A\cup \{y_i\})$.  
If $y'\notin V(P)$, then $y'y_h\dots y_1y_{h+1}\dots y_l$ is longer than $P$ which is a contradiction.  Otherwise, by \eqref{dj}, $y'\in B$.  However, $y_h\dots y_1y_{h+1}\dots y_l$ is a path for which $j-i$ is smaller, contradicting the minimality of $j-i$.

Now suppose $G[T_1]$ is not connected.  Since $\delta(G[T_1])\geq m$ and $|T_1|< 3m$, $G[T_1]$ has exactly two components.  Call these components $A$ and $C$, then set $A':=A$ and $C':=C$.  Without loss of generality, suppose $|A|\geq |C|$.  Since $\delta(G[T_1])\geq m$, we have $m+1\leq |C|$ which implies $|A|<k$, by \eqref{eq:m} and the fact that $|T_1|=2k+\floor{\alpha k}-|T_0|$.  So regardless of whether $G[T_1]$ is connected or not, all of the following hold: \eqref{ACbounds}, \eqref{dj}, \eqref{A'C'}, and  \begin{equation}\label{eq:ACnonneighbors} \forall a\in A', \overline{\|a, A\|}\leq|A|-m 
~\text{ and }~ \forall c\in C', \overline{\|c, C\|}\leq|C|-m.\end{equation}
For $Y\in\{A, C\}$, let $Y'=A'$ if $Y=A$ and let $Y'=C'$ if $Y=C$.
\begin{claim}\label{Y'}
For all $v\in V\setminus (A\cup C)$, there exists $Y\in\{A, C\}$ such that for all $y\in Y'$, $|(N(v)\cap N(y))\cap Y|\geq 3$.
\end{claim}

\begin{proof}
For all $v\in V\setminus (A\cup C)$, we have 
\begin{equation}
\|v, A\cup C\|\geq 2k-(|V|-(|A|+|C|))=|A|+|C|-k.
\end{equation}
Suppose there exists $v\in V\setminus (A\cup C)$ and $c\in C'$ such that $|(N(v)\cap N(c))\cap C|\leq 2$.  This implies that $\|v, C\|\leq |C|-m+2$ by \eqref{eq:ACnonneighbors}.  So we have $$\|v, A\|\geq |A|+|C|-k-(|C|-m+2)= |A|+m-k-2.$$  Let $a\in A'$, then by \eqref{eq:m}, $$|(N(v)\cap N(a))\cap A|\geq (|A|+m-k-2)+m-|A|= 2m-k-2\geq \frac{1}{3}k\geq 3.$$

\end{proof}

\begin{claim}
There exist two disjoint square $P^{5}$'s connecting edges of $A$ to edges of $C$.\label{con}
\end{claim}

\begin{proof}
Set $s:=\lfloor\frac{|A|}{2}\rfloor$ and $t:=\lfloor\frac{|C|}{2}\rfloor$.
Choose nonadjacent vertices $x,x'\in S$ and $a_{2s},c_{1}\in N(x)$
with $a_{2s}\in A'$ and $c_{1}\in C'$. Since $a_{2s}$ and $c_{1}$
are nonadjacent they have at least $k+1$ common neighbors distinct
from $x$, and these common neighbors are not in $A\cup C$. One of
them $v$ must also be adjacent to $x$. By Claim \ref{Y'} there exists, without loss of generality, $a_{2s-1}\in A$ such that $a_{2s},v\in N(a_{2s-1})$.  
Since $x\in S$, there exists $c_{2}\in C$ such that $x,c_{1}\in N(c_{2})$. Thus
$Q:=a_{2s-1}a_{2s}vxc_{1}c_{2}$ is a square $P^{5}$ connecting $a_{2s-1}a_{2s}$
to $c_{1}c_{2}$. Similarly, we can choose $a_{1},c_{2t}\in N(x')$
with $a_{1}\in A'-a_{2s-1}-a_{2s}$ and $c_{2t}\in C'-c_{1}-c_{2}$.
Since $a_{1}$ and $c_{2t}$ are nonadjacent, there exist $k$ common
neighbors of $a_{1}$ and $c_{2t}$ that are distinct from $x'$ and
$v$. One of them $v'$ is adjacent to $x'$, and $v'\ne x$ by the
choice of $x,x'$. Moreover, $v'\notin A\cup C$. So as above, we can choose $a_{2}\in A$ and $c_{2t-1}\in C$ so that
$Q':=c_{2t-1}c_{2t}\{v'x'\}a_{1}a_{2}$, $Q\cap Q'=\emptyset$ and $Q'$ is a square $P^{5}$
connecting $c_{2t-1}c_{2t}$ to $a_{1}a_{2}$ (note that we cannot specify the order of $v'$ and $x'$). 
\end{proof}

Finally we claim that there exist paths $$R:=a_{1}a_{2}\dots a_{2s-1}a_{2s}
\subseteq G[A] \textrm{ and } R':=c_{1}c_{2}\dots c_{2t-1}c_{2t}\subseteq G[C],$$ such that $|R|=2s$ and $|R'|=2t$. If $|A|=m$, then $A=A'$ 
and thus $G[A]$ is complete by \eqref{eq:ACnonneighbors}.
Otherwise $|A|\geq m+1$ and thus by \eqref{eq:m} we have
\begin{equation}\label{halfA}
\frac{1}{3}k+1\leq \frac{1}{2}|A|.
\end{equation}
By \eqref{dj} and \eqref{halfA}, we have 
$$\delta(G[A])\geq 2k-(|V|-(|A|+|C|))=|A|+|C|-k\geq |A|+2-(\frac{k}{3}+1)\geq \frac{1}{2}|A|+2.$$ 
Thus for all $a,a',a''\in A$, 
\begin{equation*}G[A\setminus \{a,a',a''\}]
\text{ is hamiltonian connected,}
\end{equation*} 
since $\delta(G[A\setminus\{a,a',a''\}])\geq \frac{1}{2}|A|-1>\frac{1}{2}(|A|-3)$. If $|A|=2s$, then we use the fact that $G[A\setminus\{a_1,a_{2s}\}]$ is hamiltonian connected to get $R$.  If $|A|=2s+1$ we let $a'\in A\setminus\{a_1,a_2,a_{2s-1}, a_{2s}\}$, and we use the fact that $G[A\setminus\{a_1, a_{2s}, a'\}]$ is hamiltonian connected to get $R$.  Since $|A|\geq |C|$, the same argument gives us $R'$ in $G[C]$. 

So by Claim \ref{con}, $D:=RQR'Q'$ is an even cycle of length $2s+2t+4\geq 2m+2$ (note that $D\not\subseteq G[T_1]$). Recall that $V(D)\cap S\subseteq\{x,v,x',v'\}$ and set $S':=S\setminus D$. As in Case 1, let $P:=\{p_{1},\dots,p_{s}, p_1',\dots, p_t'\}$
be a set of \emph{ports}, where $p_{i}:=\{a_{2i-1},a_{2i},a_{2i+1},a_{2i+2}\}$ for $1\leq i\leq s-1$ and $p_j':=\{c_{2j-1},c_{2j},c_{2j+1},c_{2j+2}\}$ for $1\leq j\leq t-1$. 
  By Claim \ref{ports}, 
  there exist $x_1,\dots,x_{s-1},x'_1,\dots,x'_{t-1}$ such that
    $$a_{1}a_{2}x_{1}a_{3}a_{4}x_{2}\dots x_{s-1} a_{2s-1}a_{2s}vxc_{1}c_{2}x_1'c_3c_4x_2'\dots x_{t-1}'c_{2t-1}c_{2t}\{v'x'\}a_1a_2$$ is a square cycle of length at least $2s+2t+4+s-1+t-1\geq 3m-1>2k$. Thus by Theorem~\ref{thm:FK3}, $G$ has a hamiltonian square
cycle.

\end{proof}

\section{Conclusion}
We have established a concrete threshold $n_0:=\n$ such that P\'osa's Conjecture holds for all graphs of order at least  $n_0$, using methods essentially from prior to 1996. It seems in retrospect, that we were blinded by the brilliance of the Regularity-Blow-up method, and missed that the crucial idea of \cite{KSSp} was just to divide the problem into extremal and non-extremal cases.  However P\'osa's Conjecture remains open. We suspect that our probabilistic methods cannot be used to obtain an improvement of more than a factor of 1000. On the other hand we believe that ordinary graph theoretic methods have not yet been exhausted.

We have also developed the method of special reservoirs, for removing regularity from certain arguments.  We believe that this could be used on other problems.
The paper \cite{LSS} was written with the goal of developing methods for a more general set of problems. 
In particular they used an
\emph{absorbing path} lemma which contributes to a much larger value of
$n_0$. However other problems do not (yet) have an analog of Theorem \ref{thm:FK3}, while the absorbing technique is quite adaptable. 
Here are some other possible candidates
for applying these new techniques, the first of which was discussed in
\cite{LSS}.

\begin{conjecture}[Seymour \cite{S}]
For all positive integers $k$, every graph $G$ with  $\delta(G)\ge \frac{k}{k+1}|G|$ contains the $k^\text{th}$ power of a hamiltonian cycle.
\end{conjecture}

Koml\'os, S\'ark\"ozy and Szemer\'edi \cite{KSSps,KSSs} used the Regularity and Blow-up Lemmas to prove that there exists a function $n(k)$ such that Seymour's Conjecture holds for all $k$ and graphs of order at least $n(k)$.

Ch\^au also used the Regularity and Blow-up Lemmas to prove the following Ore-type version of P\'osa's Conjecture for graphs of large order.
\begin{thm}[Ch\^au \cite{C}]
\label{Main Theorem} Let $G$ be a graph on $n$ vertices such that $d(x)+d(y)\geq
\frac{4}{3}n-\frac{1}{3}$ for all $xy\notin E(G)$.

\textup{(a)} If $\delta(G)=\frac{1}{3}n+2$ or $\delta(G)=\frac{1}{3}n+\frac{5}{3},$ then
$G$ contains a hamiltonian square path.

\textup{(b)} If $\delta(G)>\frac{1}{3}n+2,$ then for sufficiently large $n,$ $G$
contains a hamiltonian square cycle.
\end{thm}

For a directed graph $G$, the \emph{minimum semi-degree} of $G$, denoted $\delta^0(G)$, is the minimum of the minimum in-degree $\delta^-(G)$ and the minimum out-degree $\delta^+(G)$.  An \emph{oriented graph} is a directed graph with no $2$-cycles. Keevash, K\"uhn, and Osthus proved the following oriented version of Dirac's theorem using the Regularity-Blow-up method (with a directed version of the Regularity Lemma).

\begin{thm}[Keevash, K\"uhn, Osthus \cite{KKO}] Let $G$ be an oriented graph on $n$ vertices.  If $\delta^0(G)\geq \frac{3n-4}{8}$ and $n$ is sufficiently large, then $G$ contains a hamiltonian cycle.
\end{thm}

Finally Treglown conjectured the following oriented version of P\'osa's conjecture.

\begin{conjecture}[Treglown \cite{T}]
Let $G$ be an oriented graph on $n$ vertices.  If $\delta^0(G)\geq \frac{5n}{12}$, then $G$ contains a the square of a hamiltonian cycle.
\end{conjecture}

\subsection*{Acknowledgements}

We thank the referees for their thoughtful comments which improved the presentation of this paper.

\end{document}